\documentclass[leqno,11pt]{amsart}
\pdfoutput=1		

\usepackage[bookmarksnumbered]{hyperref}	
\usepackage[ngerman,USenglish]{babel}			
\usepackage{color}											
\usepackage[ascii]{inputenc}								
\usepackage{aliascnt}										
\usepackage{microtype}									
\usepackage[all]{hypcap}
\usepackage{verbatim}

\newcommand{\newjointcountertheorem}[3]{
	\newaliascnt{#1}{#2}
	\newtheorem{#1}[#1]{#3}
	\aliascntresetthe{#1}	
}

\newtheorem*{thm-galois}{Theorem \ref{thm:gen galois values}}

\newtheorem{thm}{Theorem}[section]
\newjointcountertheorem{satz}{thm}{Satz}
\newjointcountertheorem{lem}{thm}{Lemma}
\newjointcountertheorem{cor}{thm}{Corollary}
\newjointcountertheorem{prp}{thm}{Proposition}
\newjointcountertheorem{cnj}{thm}{Conjecture}
\newjointcountertheorem{que}{thm}{Question}
\newjointcountertheorem{fct}{thm}{Fact}
\theoremstyle{definition}
\newjointcountertheorem{dfn}{thm}{Definition}
\newjointcountertheorem{ntn}{thm}{Notation}
\newjointcountertheorem{rem}{thm}{Remark}
\newjointcountertheorem{nte}{thm}{Note}
\newjointcountertheorem{exl}{thm}{Example}

\newcommand{\qmultinom}[2]{\genfrac{[}{]}{0pt}{}{#1}{#2}}

\def\Snospace~{\S{}}

\numberwithin{equation}{section}
\allowdisplaybreaks[4]										

\hypersetup{pdfauthor={Stavros Kousidis},pdftitle={Asymptotics of generalized Galois numbers via affine Kac-Moody algebras}}

\begin{document}

	\title[Asymptotics of Galois numbers via affine Kac-Moody algebras]{Asymptotics of generalized Galois numbers via affine Kac-Moody algebras}
	
	\author[Stavros Kousidis]{Stavros Kousidis}
  	\address{Stavros Kousidis, Institute for Theoretical Physics\\ETH Zurich\\Wolfgang--Pauli--Strasse 27\\CH-8093 Zurich\\Switzerland}
	\email{\href{mailto:st.kousidis@googlemail.com}{st.kousidis@googlemail.com}}

	
	\begin{abstract}
		Generalized Galois numbers count the number of flags in vector spaces over finite fields. Asymptotically, as the dimension of the vector space becomes large, we give their exponential growth and determine their initial values. The initial values are expressed analytically in terms of theta functions and Euler's generating function for the partition numbers. Our asymptotic enumeration method is based on a Demazure module limit construction for integrable highest weight representations of affine Kac-Moody algebras. For the classical Galois numbers, that count the number of subspaces in vector spaces over finite fields, the theta functions are Jacobi theta functions. We apply our findings to the asymptotic number of linear $q$-ary codes, and conclude with some final remarks about possible future research concerning asymptotic enumerations via limit constructions for affine Kac-Moody algebras and modularity of characters of integrable highest weight representations.
	\end{abstract}

	\maketitle
	\tableofcontents

	\section{Introduction}	
		The generalized Galois numbers $G_N^{(r)}(q)$ count the number of flags $0 = V_0 \subseteq V_1 \subseteq \cdots \subseteq V_r = \mathbf{F}_q^N$ of length $r$ in an $N$-dimensional vector space over a field with $q$ elements \cite{v10}. In particular, when $r=2$ these are the classical Galois numbers studied by Goldman and Rota \cite{MR0252232} which give the total number of subspaces in $\mathbf{F}_q^N$.
		
		We show that the generalized Galois numbers grow asymptotically, as $r$ is fixed and $N \rightarrow \infty$, exponentially with factor $O(N^2)$ in logarithmic ``time'' scale:
		\[
			G_N^{(r)}(q) \sim I_r(q)  \cdot e^{O(N^2) \log (q)}
			.
		\]
		Here, ``time'' equals the cardinality of the finite field. Our main result is the explicit description of the initial values $I_r(q)$ via theta functions and Euler's generating function for the partition numbers.
		
		This investigation serves three purposes.
		First, the generalized Galois numbers are of independent interest as they enumerate points in fundamental geometric objects defined over finite fields. For example, by definition the classical Galois numbers
		\[
			G_N(q)=G_N^{(2)}(q) = \sum_{k=0}^N |\mathrm{Gr}(k,N)(\mathbf{F}_{q})|
		\]
		count the number of $\mathbf{F}_q$-rational points in Grassmann varieties. The numbers of solutions of the set of equations for $\mathrm{Gr}(k,N)$ in extension fields $\mathbf{F}_{p^n}$ of $\mathbf{F}_{p}$ are in turn subject to the study of local zeta-functions $Z(\mathrm{Gr}(k,N),t) = \exp ( \sum_{n \geq 1} |\mathrm{Gr}(k,N)(\mathbf{F}_{p^n})| \frac{t^n}{n} )$ in number theory. Let us mention that a generating function for the local zeta-function $Z(\mathrm{Gr}(k,N),t)$ can be given by
		\[
			Z(\mathrm{Gr}(k,N),t) = \frac{1}{(1-t)^{b_0}(1-pt)^{b_1}\ldots(1-p^{k(N-k)}t)^{b_{k(N-k)}}} ,
		\]
		where the $b_i = \dim H_{2i}(\mathrm{Gr}(k,N)(\mathbf{C}),\mathbf{Z})$ are the even topological Betti numbers of the complex Grassmannian. Consequently, the study of Galois numbers reflects upon many subjects.

		Second, the Galois numbers enumerate asymptotically the number of equivalence classes of linear $q$-ary codes in algebraic coding theory as recently shown by Hou and Wild \cite{MR2177491,MR2307131,MR2492098,MR1755766,MR2191288}. For example, the asymptotic number $N_{n,q}^{\mathfrak{S}}$ of linear $q$-ary codes under permutation equivalence is
		\[
			N_{n,q}^{\mathfrak{S}} \sim \frac{G_n(q)}{n!}
			.
		\]
		We apply our findings to those asymptotic equivalences, and derive considerable simplifications of the asymptotic enumeration of linear $q$-ary codes (\autoref{sec:applications}).
		
		Third, our investigation serves the demonstration of the asymptotic enumeration method itself (\autoref{sec:asymptotics}). We identify the generalized Galois numbers $G_N^{(r)}(q)$ as the basic specialization of the Demazure modules $V_{-N\omega_1}(\Lambda_0)$ of the affine Kac-Moody algebra $\widehat{\mathfrak{sl}}_r$ (see \eqref{demazure basic specialization galois}).
		 Those characters pass via a graded limit construction \cite{MR2323538,MR894387,MR932325,MR980506} to the characters of the fundamental representations of our affine Kac-Moody algebra:
		\[
			\lim_{n \rightarrow \infty} \chi (V_{-(rn+j)\omega_1}(\Lambda_0)) = \chi (V(\Lambda_j))
			.
		\]
		By a symmetry argument, Kac's \cite{MR513845} character formula for the basic representation
		\[
			\chi (V(\Lambda_0)) = \sum_{k=0}^\infty p^{(r-1)}(k) e^{\Lambda_0 -k\delta} \cdot \sum_{\gamma \in Q} e^{- (h || \gamma ||^2\delta + \gamma )}
		\]
		then allows us to prove our main result:
		\begin{thm-galois}
			Consider the generalized Galois number $G_N^{(r)}(q)$. For any prime power $e^\delta= p^m$ (in fact for any complex number $e^\delta$ where $\delta \in -2\pi i \mathbf{H}$) and $0 \leq j <r$ we have the limit
			\begin{align*}
				\label{eq:gen galois values}
				\lim_{n \rightarrow \infty} G^{(r)}_{rn + j}(e^\delta) \cdot e^{- u_j(r,n) \delta} & = \frac{ \Theta_{F_j} (-\frac{\delta}{2 \pi i})}{\phi(e^{-\delta})^{r-1}} .
			\end{align*}
			Here, $\phi(x)^{-1} = \prod_{m=1}^\infty (1-x^{m})^{-1}$ denotes Euler's generating function for the partition numbers, and $\Theta_{F_j}(z) = \sum_{{\bf k} \in \mathbf{Z}^{r-1}} e^{2 \pi i z F_j({\bf k})} $ are theta functions associated to the quadratic forms $F_0 , F_1 , \ldots , F_{r-1}$ on the lattice $\mathbf{Z}^{r-1}$ given by
			\begin{align*}
				F_0 (k_1,\ldots , k_{r-1}) & = \sum_{l=1}^{r-1} k_l^2 - \sum_{l=1}^{r-2} k_l k_{l+1} , \\
				F_j (k_1,\ldots , k_{r-1}) & = \left( k_j + \frac 12 \right)^2 + \sum_{l=1, l \neq j}^{r-1} k_l^2 - \sum_{l=1}^{r-2} k_l k_{l+1} .
			\end{align*}
			The exponents $u_0,u_1,\ldots ,u_{r-1}$ are
			\begin{align*}
				u_0(r,n) & = \frac{r(r-1)n^2}{2} , \\
				u_j(r,n)  & =  \frac{(rn+j)(rn+j-1)}{2} - \frac{rn(rn+2j-r)}{2r} + \frac 14
				.
			\end{align*}
		\end{thm-galois}
			For the classical Galois numbers our theta functions turn out to be Jacobi theta functions (see \autoref{galois values}).
			
			Let us conclude the introduction with the following remark on our asymptotic enumeration method. In the case of generalized Galois numbers we do not make use of the modularity of characters of integrable highest weight modules, since the prime powers $p^{-m} < 1$ lie in the region of convergence of our modular forms. However, we will discuss, in \autoref{sec:conclusion}, an important eventual application of our asymptotic enumeration method where modularity has to be exploited.
							
	\section{Notation and Background}
	\label{sec:notation}
		
		 The generalized Galois number $G_N^{(r)}(q) \in \mathbf{N}[q]$ can be defined as the specialization of the generalized $N$-th Rogers-Szeg\H{o} polynomial at $({\bf 1},q)$ \cite{v10}:
		\[
			G_N^{(r)}(q) = H_N^{(r)}({\bf 1},q)
			.
		\]
		The $N$-th generalized Rogers-Szeg\H{o} polynomial $H_N^{(r)}({\bf z},q) \in \mathbf{N}[z_1,\ldots ,z_{r},q]$ \cite{r1893a,r1893b,s1926} (see \cite{MR1634067} for an account) is defined as the generating function of the $q$-multinomial coefficients:
	\[
		H^{(r)}_N({\bf z} , q)
		= \sum_{\substack{{\bf k} = (k_1,\ldots,k_r) \in \mathbf{N}^{r}\\k_1 + \ldots + k_r =N}}\qmultinom{N}{{\bf k}}_q {\bf z^k}
		.	
	\]	
	Recall from \cite{v10} that the $q$-multinomial coefficient $\qmultinom{N}{k_1, \ldots , k_r}_q$ counts the number of flags $0 = V_0 \subseteq \cdots \subseteq V_r = \mathbf{F}_q^N$ subject to the conditions $\dim(V_i) = k_1 + \cdots + k_i$.

		For general facts about affine Kac-Moody algebras and their representation theory we refer the reader to \cite{MR2188930,MR1104219}, and for Demazure modules to \cite{MR2323538}.
		Let us briefly fix the notation we will use throughout.
		We consider the affine Kac-Moody algebra $\widehat{\mathfrak{sl}}_{r}$. We denote the simple roots by $\alpha_0 , \alpha_1, \ldots , \alpha_{r-1}$, the highest root by $\theta = \alpha_1 + \ldots + \alpha_{r-1}$ and the imaginary root by $\delta = \alpha_0 + \theta$. The affine root lattice is then defined as $\widehat{Q} = \mathbf{Z} \alpha_0 \oplus \mathbf{Z} \alpha_1 \oplus \ldots \oplus \mathbf{Z} \alpha_{r-1}$ and the real span of the simple roots is given by $\widehat{\mathfrak{h}}_{\mathbf{R}}^\ast = \mathbf{R} \otimes_{\mathbf{Z}} \widehat{Q}$. We have a non-degenerate symmetric bilinear form on $\widehat{\mathfrak{h}}_{\mathbf{R}}^\ast$ by $\langle \alpha_i , \alpha_j \rangle = c_{ij}$ where $C = (c_{ij})$ is the Cartan matrix of $\widehat{\mathfrak{sl}}_{r}$, and define $|| \cdot ||^2 = (2h)^{-1} \langle \cdot , \cdot \rangle$ where $h=r$ is the Coxeter number of $\widehat{\mathfrak{sl}}_{r}$.
		For a dominant integral weight $\Lambda = m_1\Lambda_0 + m_2\Lambda_1 + \ldots + m_{r-1} \Lambda_{r-1}$ we let $V(\Lambda)$ be the integrable highest weight representation of weight $\Lambda$ of $\widehat{\mathfrak{sl}}_{r}$ and $\chi(V(\Lambda))$ its character. The $\Lambda_0, \Lambda_1, \ldots , \Lambda_{r-1}$ are called fundamental weights,  the $V(\Lambda_l)$ the fundamental representations and $V(\Lambda_0)$ the basic representation. 
		As for the Demazure modules, we will only consider the translations $t_{-k\omega_1} = (s_1 s_2 \ldots s_{r-1} \sigma^{r-1})^k$ in the extended affine Weyl group of $\widehat{\mathfrak{sl}}_{r}$, where $\omega_1 = \Lambda_1 - \Lambda_0$. Here, $\sigma$ denotes the automorphism of the Dynkin diagram of $\widehat{\mathfrak{sl}}_{r}$ which sends $0$ to $1$, and $s_1, \ldots , s_{r-1}$ are the simple reflections associated to the simple roots $\alpha_1, \ldots , \alpha_{r-1}$. We denote the Demazure module associated to those translations by $V_{-k\omega_1}(\Lambda)$ and its character by $\chi(V_{-k\omega_1}(\Lambda))$. We write the monomials in the characters of our modules as $e^{\lambda}$, the coefficient $k$ in the monomial $e^{-k\alpha_0}$ is referred to as the degree.
		
		$\mathbf{H}$ will denote the upper half plane in $\mathbf{C}$. We write $\sim$ for asymptotic equivalence, that is for $f,g : \mathbf{N} \rightarrow \mathbf{R}_{>0}$ we write $f(n) \sim g(n)$ if $\lim_{n \rightarrow \infty} f(n)/g(n) =1$.
	
	\section{Asymptotics of generalized Galois numbers}
	\label{sec:asymptotics}
	
	Let us start with a direct consequence of Kac's character formula \cite[(3.37)]{MR513845}.
	
	\begin{prp}
		Consider the basic representation  $V(\Lambda_0)$ of $\widehat{\mathfrak{sl}}_{r}$. Let $Q$ be the lattice $Q = \widehat{Q}/\mathbf{Z}\alpha_0 = \mathbf{Z}\alpha_1 \oplus \ldots \oplus \mathbf{Z} \alpha_{r-1} \cong \mathbf{Z}^{r-1}$. Then,
		\begin{align}
		\label{character basic representation}
			\chi (V(\Lambda_0)) & = \frac{e^{\Lambda_0}}{\phi(e^{-\delta})^{r-1}} \cdot \sum_{{\bf k} \in \mathbf{Z}^{r-1}} e^{\frac 14 \theta} \prod_{l=0}^{r-1} e^{-F_l({\bf k}) \alpha_l}
			.
		\end{align}
		Here, $\phi(x)^{-1} = \prod_{i=1}^m (1-x^m)^{-1}$ is Euler's generating function for the partition numbers, and the $F_0,F_1,\ldots ,F_{r-1}$ are quadratic forms on the lattice $Q \cong \mathbf{Z}^{r-1}$ defined as
		\begin{align}
			F_0 (k_1,\ldots , k_{r-1}) & = \sum_{l=1}^{r-1} k_l^2 - \sum_{l=1}^{r-2} k_l k_{l+1} ,\\
			F_j (k_1,\ldots , k_{r-1}) & = \left( k_j + \frac 12 \right)^2 + \sum_{l=1, l \neq j}^{r-1} k_l^2 - \sum_{l=1}^{r-2} k_l k_{l+1} .
		\end{align}
	\end{prp}
	
	\begin{proof}
		Due to Kac \cite[(3.37)]{MR513845} we have the following character formula vor the basic representation $V(\Lambda_0)$:
		\begin{align}
		\label{eq:Kac character}
			\chi (V(\Lambda_0)) = \sum_{k=0}^\infty p^{(r-1)}(k) e^{\Lambda_0 -k\delta} \cdot \sum_{\gamma \in Q} e^{- (h || \gamma ||^2\delta + \gamma )}
		\end{align}
		The function $p^{(r-1)}(k)$ is defined via its generating function $\sum_{k=0}^\infty p^{(r-1)}(k) x^k = \phi(x)^{-r+1}$. 
		Let $f({\bf k})$ be the quadratic form $f({\bf k}) = f(k_1 , \ldots , k_{r-1}) = \sum_{i=1}^{r-1} k_i^2 - \sum_{i=1}^{r-2} k_i k_{i+1}$ on the lattice $Q \cong \mathbf{Z}^{r-1}$. If we express $\gamma \in Q$ as the linear combination $\gamma = k_1 \alpha_1 + \ldots + k_{r-1} \alpha_{r-1}$, we have $h|| \gamma ||^2 = f({\bf k})$. Then,
		\begin{align*}
			\chi (V(\Lambda_0)) & = \sum_{k=0}^\infty p^{(r-1)}(k) e^{\Lambda_0 -k\delta} \cdot \sum_{\gamma \in Q} e^{- (h || \gamma ||^2\delta + \gamma )} \\
				& = \frac{e^{\Lambda_0}}{\phi(e^{-\delta})^{r-1}} \cdot \sum_{\gamma \in Q} e^{- (h || \gamma ||^2\delta + \gamma) } \\
				& = \frac{e^{\Lambda_0}}{\phi(e^{-\delta})^{r-1}} \cdot \sum_{{\bf k} \in \mathbf{Z}^{r-1}} e^{- f({\bf k}) \delta} e^{-k_1\alpha_1 } \ldots e^{-k_{r-1} \alpha_{r-1}} \\
				& = \frac{e^{\Lambda_0}}{\phi(e^{-\delta})^{r-1}} \cdot \sum_{{\bf k} \in \mathbf{Z}^{r-1}} e^{- f({\bf k}) \alpha_0} e^{-(f({\bf k}) +k_1)\alpha_1} \ldots e^{-(f({\bf k}) +k_{r-1})\alpha_{r-1}}
				.
		\end{align*}
		Note that $F_0({\bf k}) = f({\bf k})$ and $F_j ({\bf k}) - \frac 14 = f ({\bf k}) +k_j$. This finishes the proof.
	\end{proof}
	
	We are ready to prove our main result.
	
	\begin{thm}
	\label{thm:gen galois values}
		Consider the generalized Galois number $G_N^{(r)}(q)$. For any prime power $e^\delta= p^m$ (in fact for any complex number $e^\delta$ where $\delta \in -2\pi i \mathbf{H}$) and $0 \leq j <r$ we have the limit
			\begin{align}
				\label{eq:gen galois values}
				\lim_{n \rightarrow \infty} G^{(r)}_{rn + j}(e^\delta) \cdot e^{- u_j(r,n) \delta} & = \frac{ \Theta_{F_j} (-\frac{\delta}{2 \pi i})}{\phi(e^{-\delta})^{r-1}} .
			\end{align}
			Here, $\phi(x)^{-1} = \prod_{m=1}^\infty (1-x^{m})^{-1}$ denotes Euler's generating function for the partition numbers, and $\Theta_{F_j}(z) = \sum_{{\bf k} \in \mathbf{Z}^{r-1}} e^{2 \pi i z F_j({\bf k})} $ are theta functions associated to the quadratic forms $F_0 , F_1 , \ldots , F_{r-1}$ on the lattice $\mathbf{Z}^{r-1}$ given by
			\begin{align}
				F_0 (k_1,\ldots , k_{r-1}) & = \sum_{l=1}^{r-1} k_l^2 - \sum_{l=1}^{r-2} k_l k_{l+1} , \\
				F_j (k_1,\ldots , k_{r-1}) & = \left( k_j + \frac 12 \right)^2 + \sum_{l=1, l \neq j}^{r-1} k_l^2 - \sum_{l=1}^{r-2} k_l k_{l+1} .
			\end{align}
			The exponents $u_0,u_1,\ldots ,u_{r-1}$ are
			\begin{align}
				u_0(r,n) & = \frac{r(r-1)n^2}{2} , \\
				u_j(r,n)  & = \frac{(rn+j)(rn+j-1)}{2} - \frac{rn(rn+2j-r)}{2r} + \frac 14
				.
			\end{align}
	\end{thm}
		
	\begin{proof}
		Let $N = rn+j$ for $0\leq j<r$ and consider the Demazure module $V_{-N\omega_1}(\Lambda_0)$ associated to the translation $t_{-N\omega_1} = (s_1 s_2 \ldots s_{r-1} \sigma^{r-1})^N$. By Sanderson \cite{MR1771615} we can describe its character $\chi(V_{-N\omega_1}(\Lambda_0))$ via a certain specialization of a symmetric Macdonald polynomial (see \cite[Chapter VI]{MR1354144} for their definition and properties). That is, let $[N] = (N,0,\ldots,0) \in \mathbf{N}^{r}$ denote the one-row Young diagram, and $\eta_N$ the smallest composition of degree $N$, i.e.~since  $N = rn +j$ we have $\eta_N = ((n)^{r-j},(n+1)^j) \in \mathbf{N}^{r}$. Following \cite[\S 2]{MR1771615} we have $[N] = t_{-N\omega_1} \cdot \eta_N$ 	with the convention $\sigma \cdot \eta_N = \eta_N$. Furthermore, by computing the expression $u([N])-u(\eta_{[N]})$ in \cite[Theorem 6]{MR1771615}, the maximal occurring degree in $\chi(V_{-N\omega_1}(\Lambda_0))$ is given by
		\begin{align}
			d_r(N) = d_r(rn+j) = \frac{(rn+j)(rn+j-1)}{2} - \frac{rn(rn+2j-r)}{2r} 
			.
		\end{align}
		Note that $d_r(rn) = u_0 (r,n)$ and $d_r(rn+j) = u_j (r,n) - \frac 14$ for $j = 1,\ldots ,r-1$.
		Let ${\bf z} = (e^{\Lambda_1 - \Lambda_0}, e^{\Lambda_2 -\Lambda_1}, \ldots , e^{\Lambda_{r-1} - \Lambda_{r-2}}, e^{\Lambda_0 -\Lambda_{r-1}})$.
		Then, by \cite[Theorem 6 and 7]{MR1771615}\footnote{There seems to be a missprint in \cite[\S 4]{MR1771615}. Namely, the image $\pi(q)$ should equal $q = e^\delta$, not $q = e^{-\delta}$.}   we have
		\begin{align}
			\label{demazure macdonald}
				\chi(V_{-N \omega_1}(\Lambda_0)) = e^{\Lambda_0-d_r(N)\delta} \cdot P_{[N]}({\bf z};e^\delta,0)
				,
		\end{align}
		where $P_{[N]}({\bf z};q,0)$ denotes the specialized symmetric Macdonald polynomial associated to the partition $[N]$. Furthermore, by Hikami \cite[Equation (3.4)]{MR1457389} this Macdonald polynomial equals the $N$-th generalized Rogers-Szeg\H{o} polynomial:
		\begin{align}
			\label{macdonald rogerszegoe}
				 P_{[N]}({\bf z};q,0) = H_N^{(r)}({\bf z} , q)
				.
		\end{align}
		Combining \eqref{demazure macdonald} and \eqref{macdonald rogerszegoe} we obtain
		\begin{align}
			\label{demazure rogersszegoe}
				\chi (V_{-N \omega_1}(\Lambda_0)) = e^{\Lambda_0-d_r(N)\delta} \cdot H_N^{(r)}({\bf z},e^\delta)
				.
		\end{align}
		Consequently, the basic specialization at $e^{-\alpha_1} = \ldots = e^{-\alpha_{r-1}} =1$ of the Demazure character on the left-hand side of \eqref{demazure rogersszegoe} gives the generalized Galois number $G_N^{(r)}$ up to translation:
		\begin{align}
			\label{demazure basic specialization galois}
				\left. \chi (V_{-N\omega_1}(\Lambda_0)) \right\vert_{(e^{-\alpha_1} = \ldots = e^{-\alpha_{r-1}} =1)} = e^{\Lambda_0-d_r(N)\delta}\cdot G_N^{(r)} (e^\delta)
				.
		\end{align}
		
		Now, let us proceed to the limit considerations. By Fourier and Littelmann \cite[Theorem D]{MR2323538} (which is based on work by Mathieu and Kumar \cite{MR894387,MR932325,MR980506}) the characters of our Demazure modules pass, as $N \rightarrow \infty$, as functions in $(e^{-\alpha_0},e^{-\alpha_1}, \ldots , e^{-\alpha_{r-1}})$ to the characters of the fundamental representations $V(\Lambda_0), V(\Lambda_1), \ldots , V(\Lambda_{r-1})$ of $\widehat{\mathfrak{sl}}_{r}$ as follows
		\begin{align}
				\label{limit demazure weyl-kac}
				\lim_{n \rightarrow \infty} \chi (V_{-(rn+j)\omega_1}(\Lambda_0)) & = \chi (V(\Lambda_j))
				.
		\end{align}
		
		We are ready to prove our claimed identity \eqref{eq:gen galois values} in the case $j=0$. Recall that $\delta = \alpha_0 + \theta$ where $\theta = \alpha_1 + \ldots + \alpha_{r-1}$, and $d_r(rn) = u_0 (r,n)$. Then, the equations \eqref{demazure basic specialization galois}, \eqref{limit demazure weyl-kac} and the character formula \eqref{character basic representation} for the basic representation $V(\Lambda_0)$ imply
			\begin{align*}
				\lim_{n \rightarrow \infty} G_{rn}^{(r)} (e^\delta) \cdot e^{- u_0(r,n)\delta}
				& = e^{-\Lambda_0} \cdot \left. \chi (V(\Lambda_0)) \right\vert_{(e^{-\alpha_1} = \ldots = e^{-\alpha_{r-1}} =1)} \\
				& = \frac{1}{\phi(e^{-\delta})^{r-1}} \cdot \sum_{{\bf k} \in \mathbf{Z}^{r-1}} e^{-F_0({\bf k}) \delta} \\
				& = \frac{1}{\phi(e^{-\delta})^{r-1}} \cdot \Theta_{F_0} \left(-\frac{\delta}{2 \pi i} \right).
			\end{align*}

			There is a subtlety to our deduction in the cases $j=1,\ldots ,r-1$. The characters of the representations $V(\Lambda_0), V(\Lambda_1), \ldots , V(\Lambda_{r-1})$ are symmetrical in the sense that they subsequently differ by an application of the automorphism $\sigma$ that sends $0$ to $1$ in the Dynkin diagram of $\widehat{\mathfrak{sl}}_{r}$. An application of $\sigma$ cyclically shifts the fundamental weights $\Lambda_i \mapsto \Lambda_{i+1}$ and the simple roots $\alpha_i \mapsto \alpha_{i+1}$ (by cyclic we mean $\Lambda_{r}=\Lambda_0$ and $\alpha_{r} = \alpha_0$). Consequently, it leaves $\delta$ invariant $\sigma(\delta) = \delta$ and $\sigma^j (\theta) = \alpha_0 + \theta - \alpha_{j}$ for $j=1,\ldots ,r-1$. To be precise, Kac's character formula \eqref{character basic representation} for the fundamental representations $V(\Lambda_1),\ldots ,V(\Lambda_{r-1})$ reads as follows
			\begin{align}
				\chi (V(\Lambda_j)) & = \frac{e^{\Lambda_j}}{\phi(e^{-\delta})^{r-1}} \cdot \sum_{{\bf k} \in \mathbf{Z}^{r-1}} e^{\frac 14 (\alpha_0 + \theta - \alpha_{j})} \prod_{l=0}^{r-1} e^{-F_{l+r-j}({\bf k}) \alpha_l}
			.
		\end{align}
		Recall that $u_j(r,n) = d_r(rn+j) + \frac 14$. Therefore, we obtain
			\begin{align*}
				\lim_{n \rightarrow \infty} G^{(r)}_{rn + j}(e^\delta) \cdot e^{- u_j(r,n) \delta} & = e^{-\frac 14 \delta} \cdot \lim_{n \rightarrow \infty} G^{(r)}_{rn + j}(e^\delta) \cdot e^{- d_r(rn+j) \delta} \\
					& = e^{-\frac 14 \delta} \cdot e^{-\Lambda_j} \cdot \left. \chi (V(\Lambda_j)) \right\vert_{(e^{-\alpha_1} = \ldots = e^{-\alpha_{r-1}} =1)} \\
					& = \frac{e^{-\frac 14 \delta}}{\phi(e^{-\delta})^{r-1}} \cdot \sum_{{\bf k} \in \mathbf{Z}^{r-1}} e^{\frac 14 \delta} e^{-F_j({\bf k})\delta} \\
					& = \frac{1}{\phi(e^{-\delta})^{r-1}} \cdot \Theta_{F_j} \left(-\frac{\delta}{2 \pi i} \right)
					.
			\end{align*}
			This establishes the theorem.
		\end{proof}
		
		\begin{rem}
		\label{intrinsic quadratic forms}
			Motivated by Kac's character formula \eqref{eq:Kac character} we write the quadratic forms on $Q = \mathbf{Z}\alpha_1 \oplus \ldots \oplus \mathbf{Z}\alpha_{r-1} \cong \mathbf{Z}^r$ intrinsically in terms of data associated to our affine Kac-Moody algebra $\widehat{\mathfrak{sl}}_{r}$ as follows. For $j =1,\ldots ,r-1$ one has
			\begin{align}
				F_0 (\gamma) & = h || \gamma ||^2 \\
				F_j (\gamma) & = h || \gamma ||^2 + \langle \Lambda_j , \gamma \rangle +\frac 14.
			\end{align}
			The exponents $u_0 , u_1 , \ldots , u_{r-1}$ can be described via the translation formula \cite[(6.5.3)]{MR1104219}.
		\end{rem}
		
		\begin{rem}
			One can phrase \autoref{thm:gen galois values} asymptotically as
			\begin{align}
				G_{rn+j}^{(r)}(e^\delta) \sim \frac{\Theta_{F_j}(-\frac{\delta}{2 \pi i})}{\phi(e^{-\delta})^{r-1}} \cdot (e^{\delta})^{u_j(n,r)}
			\end{align}
			Note that for fixed $r$ the exponents $u_j(n,r)$ lie in $O(n^2)$.
		\end{rem}
		
		\begin{rem}
			For $j=1,\ldots,r-1$ the limits in \autoref{thm:gen galois values} coincide. In fact, the quadratic forms $F_1,\ldots,F_{r-1}$ differ only by a cyclic shift of the coordinates. Summation over the complete lattice $\mathbf{Z}^{r-1}$ produces equality.
		\end{rem}
		
		Let us summarize the implications of \autoref{thm:gen galois values} for the classical Galois numbers $G_N(q) = G_N^{(2)}(q)$ that count the number of subspaces in $\mathbf{F}_q^N$.
		
		\begin{cor}
		\label{galois values}
			Consider the classical Galois numbers $G_N(q)$. For any prime power $q = p^m$ (in fact for any complex number $|q| > 1$) we have
			\begin{align}
				\label{galois values odd}
				g_{2\infty +1}(q) & = \lim_{n \rightarrow \infty} G_{2n+1}(q) \cdot q^{-\frac{(2n+1)^2 }{4}} = \frac{\vartheta_2(0,q^{-1})}{\phi(q^{-1})} , \\
				\label{galois values even}
				g_{2\infty}(q) & = \lim_{n \rightarrow \infty} G_{2n}(q) \cdot q^{- \frac{(2n)^2}{4}} = \frac{\vartheta_3(0,q^{-1})}{\phi(q^{-1})}
				.
			\end{align}
			Here, $\phi(x)^{-1} = \prod_{m=1}^\infty (1-x^{m})^{-1}$ denotes Euler's generating function for the partition numbers, and $\vartheta_2,\vartheta_3$ are the Jacobi theta functions
			\begin{align*}
				\vartheta_2(z,q) & = \sum_{k=-\infty}^{\infty} q^{(k+ \frac 12)^2} e^{(2k+1)iz} , \\
				\vartheta_3(z,q) & = \sum_{k=-\infty}^{\infty} q^{k^2} e^{2kiz}
				.
			\end{align*}
			The limits differ by
			\begin{align}
			\label{difference galois values}
				 g_{2\infty}(q) - g_{2\infty +1}(q) = \frac{\vartheta_4(0,q^{-\frac 14})}{\phi(q^{-1})} ,
			\end{align}
			where $\vartheta_4$ is the Jacobi theta function
			\begin{align*}
				\vartheta_4(z,q) = \sum_{k=-\infty}^\infty (-1)^k q^{k^2}e^{2kiz} .
			\end{align*}
		\end{cor}
		
		\begin{proof}
			It remains to prove \eqref{difference galois values}. This follows from the identity $\vartheta_4(z,q) = \vartheta_3(2z,q^4) - \vartheta_2(2z,q^4)$ \cite[pp.~464]{MR0178117}.
		\end{proof}
		
		For completeness, we take a closer look at the numbers $g_{2\infty +1}(q)$, $g_{2\infty}(q)$ and their differences.
		
		\begin{cor}
		\label{closer asymptotic galois values}
			If $q$ is a prime power (in fact $q \geq 2$)
			\begin{align}
			\label{comparison}
				g_{2\infty}(2) \geq g_{2\infty}(q) > g_{2\infty+1}(q) > 0.
			\end{align}
			For large values of $q$ we have
			\begin{align}
				\label{asymptotic large primes odd galois values}
				g_{2\infty+1}(q) \cdot q^{\frac{(2n+1)^2}{4}} & \sim 2q^{n(n-1)} ,\\
				\label{asymptotic large primes even galois values}
				g_{2\infty}(q) \cdot q^{n^2} & \sim q^{n^2} ,
			\end{align}
			and consequently
			\begin{align}
				\label{large primes odd galois values}
				\lim_{q \rightarrow \infty} g_{2\infty+1}(q) & = 0 , \\
				\label{large primes even galois values}
				\lim_{q \rightarrow \infty} g_{2\infty}(q) & = 1 .
			\end{align}
		\end{cor}
		
		\begin{proof}
			Our statements can be deduced from the Jacobi triple product identity. Namely, we have $\vartheta_4(0,q^{-1}) = \prod_{m=1}^\infty (1-q^{-2m})(1-q^{-(2m-1)})^2$ which is $>0$ if $q>1$. Since $\phi(q^{-1})>0$ for $q>1$, the strict inequality $g_{2\infty}(q) > g_{2\infty+1}(q)$ is established. For the obviously sharp bound $g_{2\infty}(2) \geq g_{2\infty}(q)$ we look at
		 \begin{align}
		 \label{triple product galois even}
		 	\frac{\vartheta_3(0,q^{-1})}{\phi(q^{-1})} & = \frac{\prod_{m=1}^\infty (1-q^{-2m}) (1+q^{-(2m-1)})^2}{\prod_{m=1}^\infty (1-q^{-m})} \\
				\notag
				& = \prod_{m=1}^\infty (1+q^{-m})(1+q^{-(2m-1)})^2
		 \end{align}
		 which is a product of compositions of monotonic functions on $q > 1$. This identity also shows \eqref{asymptotic large primes even galois values} and \eqref{large primes even galois values}. To prove $g_{2\infty+1}(q) > 0$, \eqref{asymptotic large primes odd galois values} and \eqref{large primes odd galois values} one considers
		 \begin{align}
		 \label{triple product galois odd}
		 	\frac{\vartheta_2(0,q^{-1})}{\phi(q^{-1})} & = \frac{ 2q^{-\frac 14} \prod_{m=1}^\infty (1-q^{-2m}) (1+q^{-2m})^2}{\prod_{m=1}^\infty (1-q^{-m})} \\
				\notag
				& = 2 q^{-\frac 14} \prod_{m=1}^\infty (1+q^{-m})(1+q^{-2m})^2 .
				\qedhere
		 \end{align}
		\end{proof}
		
		For some prime powers $q$, the numbers $g_{2\infty+1}(q)$, $g_{2\infty}(q)$ and their differences have been listed in \autoref{table galois numbers}. The table has been produced in Mathematica with the following functions (up to a $10$ digit precision: \verb+N[ ,10]+) for $g_{2\infty+1}(q)$, $g_{2\infty}(q)$ and $g_{2\infty}(q) - g_{2\infty+1}(q)$, respectively.
			\[
			\begin{tabular}{l}
				\verb+f[q_]:=N[EllipticTheta[2,0,1/q]1/QPochhammer[1/q,1/q],10]+ \\ 
				\verb+g[q_]:=N[EllipticTheta[3,0,1/q]1/QPochhammer[1/q,1/q],10]+ \\ 
				\verb+h[q_]:=N[EllipticTheta[4,0,q^(-1/4)]1/QPochhammer[1/q,1/q],10]+
			\end{tabular}
			\]
			
			Certainly, our \autoref{thm:gen galois values} allows an implementation for evaluating the asymptotic initial values of generalized Galois numbers.
		
			\begin{table}
			\caption{Asymptotic Galois numbers}
			\label{table galois numbers}
			\begin{tabular}{cccc}
				$q$ & $g_{2\infty+1}(q)$ & $g_{2\infty}(q)$ & $g_{2\infty}(q) - g_{2\infty+1}(q)$ \\
				&&& \\
				$2$ & $7.371949491$ &	 $7.371968801$ & $0.0000193107$ \\
				$3$ & $3.018269046$ & $3.019783846$ & $0.0015147993$ \\
				$5$ & $1.829548122$ & $1.845509008$ & $0.0159608865$ \\
				$7$ & $1.499386995$ &   $1.537469387$ & $0.0380823915$ \\
				$11$ & $1.229171217$ & $1.312069129$ & $0.0828979124$ \\
				$13$ & $1.155207999$ & $1.258137150$ & $0.1029291515$ \\
				$17$ & $1.054013475$ & $1.191906557$ & $0.1378930825$ \\
				$19$ & $1.016940655$ & $1.170103722$ & $0.1531630663$ \\
				$23$ & $0.9584786871$ & $1.138621162$ & $0.1801424752$ \\
				$29$ & $0.8947912163$ & $1.108510891$ & $0.2137196747$ \\
				$29^{2011}$ & $1.203473556 \cdot 10^{-735}$ & $1.000000000$ & $1.0000000000$
			\end{tabular}	
			\end{table}
	
		\begin{rem}
			\autoref{galois values} can be derived by the character formula of Feingold and Lepowsky \cite[Theorem 4.5]{MR509801} for the basic representation of $\widehat{\mathfrak{sl}}_2$. That is,
			\begin{align}
			\label{feingold lepowsky 0}
				\chi (V(\Lambda_0)) = \sum_{k=0}^\infty p(k) e^{\Lambda_0 - k\delta} \sum_{l=-\infty}^\infty e^{-l^2 \alpha_0} e^{-l(l+1)\alpha_1}
				,
			\end{align}
			where $p(k)$ is the partition function that counts the number of ways to write $k$ as a sum of positive integers. In fact, Kac's character formula \cite[(3.37)]{MR513845} reduces to this expression (see \cite[(3.39)]{MR513845}), and our proof of \autoref{thm:gen galois values} reduces to this setting.
		\end{rem}
		
		\section{Applications to linear $q$-ary codes}
		\label{sec:applications}
		
		To describe the asymptotic number of non-equivalent binary $n$-codes in terms of the classical Galois numbers $G_n(2)$, Wild \cite{MR1755766, MR2191288} examines numbers $d_1(q),d_2(q)$ (see Lemma 1 in both articles) which, in the notation of \autoref{galois values}, are defined as 
		\begin{align*}
			d_1(q) & = g_{2\infty +1}(q) , \\
			d_2(q) & = g_{2\infty}(q) .
		\end{align*}
		He proves that they are positive constants (depending on $q$) less than $32$, gives a numerical evaluation method by use of the recursion formula of Goldman and Rota \cite[(5)]{MR0252232}, evaluates $d_1(q)$, $d_2(q)$ numerically for $q=2$, and shows $d_1(q) < d_2(q)$ for general $q$. Now, the detailed analytic behavior of those numbers can be extracted from \autoref{galois values} and \autoref{closer asymptotic galois values} (see also \autoref{table galois numbers} for examples).
		
		For a general prime power $q$, Hou \cite{MR2177491,MR2492098} derives asymptotic equivalences for the numbers of linear $q$-ary codes under three notions of equivalence. That is, the permutation equivalence ($\mathfrak{S}$), the monomial equivalence ($\mathfrak{M}$), and semi-linear monomial equivalence ($\Gamma$). He proves
		\begin{align*}
			N_{n,q}^{\mathfrak{S}} & \sim \frac{G_n(q)}{n!} , \\
			N_{n,q}^{\mathfrak{M}} & \sim \frac{G_n(q)}{n!(q-1)^{n-1}} , \\
			N_{n,q}^{\Gamma} & \sim \frac{G_n(q)}{n!(q-1)^{n-1}a} ,
		\end{align*}
		where $a = \left| \mathrm{Aut}(\mathbf{F}_q) \right| = \log_p(q)$ with $p = \mathrm{char}(\mathbf{F}_q)$. The asymptotic equivalence $N_{n,2}^{\mathfrak{S}} \sim \frac{G_n(2)}{n!}$ concerns binary codes and is previously derived by Wild \cite{MR1755766, MR2191288}. Based on their results, the transitivity of $\sim$ and our \autoref{galois values} produce the following list.
		\begin{cor}
		\label{q-ary codes asymptotic}
			The asymptotic numbers of linear $q$-ary codes, as $q$ is fixed and $n \rightarrow \infty$, under the three notions of equivalence $(\mathfrak{S})$, $(\mathfrak{M})$ and $(\Gamma)$ are given by
			 \begin{align}
			 	N_{2n+1,q}^{\mathfrak{S}} & \sim \frac{\vartheta_2(0,q^{-1})}{\phi(q^{-1})} \cdot \frac{q^{\frac{(2n+1)^2}{4}}}{(2n+1)!} , \\
				N_{2n,q}^{\mathfrak{S}} & \sim \frac{\vartheta_3(0,q^{-1})}{\phi(q^{-1})} \cdot \frac{q^{n^2}}{(2n)!} , \\
				N_{2n+1,q}^{\mathfrak{M}} & \sim \frac{\vartheta_2(0,q^{-1})}{\phi(q^{-1})} \cdot  \frac{q^{\frac{(2n+1)^2}{4}}}{(2n+1)!(q-1)^{2n}} , \\
				N_{2n,q}^{\mathfrak{M}} & \sim \frac{\vartheta_3(0,q^{-1})}{\phi(q^{-1})} \cdot  \frac{q^{n^2}}{(2n)!(q-1)^{2n-1}} , \\
				N_{2n+1,q}^{\Gamma} & \sim \frac{\vartheta_2(0,q^{-1})}{\phi(q^{-1})} \cdot  \frac{q^{\frac{(2n+1)^2}{4}}}{(2n+1)!(q-1)^{2n}a} , \\
				N_{2n,q}^{\Gamma} & \sim \frac{\vartheta_3(0,q^{-1})}{\phi(q^{-1})} \cdot  \frac{q^{n^2}}{(2n)!(q-1)^{2n-1}a} .
			\end{align}
			Furthermore, for large prime powers $q$ one has
			\begin{align}
				\frac{\vartheta_2(0,q^{-1})}{\phi(q^{-1})} \cdot q^{\frac{(2n+1)^2}{4}} & \sim 2q^{n(n-1)} , \\
				\frac{\vartheta_3(0,q^{-1})}{\phi(q^{-1})} \cdot q^{n^2} & \sim q^{n^2} .
			\end{align}
		\end{cor}
		
		For the last two statements see \autoref{closer asymptotic galois values}.
		
		\section{Conclusion}
		\label{sec:conclusion}
			The asymptotic enumeration method presented in this article can be summarized as follows. Once a certain specialization of Demazure characters has been identified with an interesting combinatorial function, the limit construction for affine Kac-Moody algebras can be used to carry it along towards the character of the integrable highest weight module, and derive asymptotic identities. There are at least two bottlenecks that one has to pass. First, a suitable character formula (for the limiting integrable highest weight representation) has to be available, that performs well with the chosen specialization. Fortunately, there is a great number of results and literature available, e.g.~\cite{MR509801,MR513845,MR585190,MR750341} (see also \cite{MR2534107,MR2719689,MR1810948}). Second, the domain, of the combinatorial function, that enumerates the objects in question must lie in the region of convergence of the limiting expressions. For example, Demazure modules specialize to tensor products of representations of the underlying finite-dimensionsal Lie algebra. Unfortunately, the analytic string functions limiting the tensor product multiplicities cannot be simply evaluated, for reasons of (non-)convergence, at the value $1$. A much finer analysis of their asymptotic behavior when $q \rightarrow 1$ is needed, that has to exploit the fact that we deal with modular forms \cite{MR750341}. Such an asymptotic analysis must take the maximal weights in the integrable highest weight module into account where those string functions emerge. Possibly borrowing and mimicking terminology from stochastic analysis like the central limit region, moderate and strong deviations region, and region of rare events. An investigation of tensor product multiplicities along those lines is planned in a future publication.
			
			An interesting alternative project could be to re-interpret our asymptotic enumeration method geometrically through the geometric realization of Demazure and integrable highest weight modules via cohomology of Schubert and flag varieties \cite{MR1923198}.
		
	\section{Acknowledgements}
		
		This work was supported by the Swiss National Science Foundation (grant PP00P2-128455), the National Centre of Competence in Research `Quantum Science and Technology', and the German Science Foundation (SFB/TR12, and grants \mbox{CH~843/1-1}, \mbox{CH~843/2-1}).
		
		I would like to thank Matthias Christandl for his kind hospitality at the ETH Zurich, Peter Littelmann who drew my attention towards the Demazure module limit construction a couple of years ago, Thomas Bliem who pointed me towards Rogers-Szeg\H{o} polynomials, and Ghislain Fourier for many helpful conversations.
	
	\bibliographystyle{amsplain}
	\bibliography{galoiskacmoody}
	
\end{document}